\documentclass[10pt,leqno]{amsart}
\usepackage{indentfirst, amssymb,amsmath,amsthm}
\newtheorem*{theoA}{Theorem A}
\newtheorem*{theoB}{Theorem B}
\newtheorem*{theoC}{Theorem C}
\newtheorem*{theoD}{Theorem D}
\newtheorem*{theoE}{Theorem E}
\newtheorem*{theoF}{Theorem F}
\newtheorem*{theoG}{Theorem G}
\newtheorem*{theoH}{Theorem H}

\newtheorem{theo}{Theorem}[section]
\newtheorem{lem}{Lemma}[section]

\newtheorem{exm}{Example}[section]
\newtheorem{defi}{Definition}[section]
\newtheorem{rem}{Remark}[section]
\newcommand{\ol}{\overline}
\newcommand{\be}{\begin{equation}}
\newcommand{\ee}{\end{equation}}
\newcommand{\beas}{\begin{eqnarray*}}
\newcommand{\eeas}{\end{eqnarray*}}
\newcommand{\bea}{\begin{eqnarray}}
\newcommand{\eea}{\end{eqnarray}}
\newcommand{\lra}{\longrightarrow}
\numberwithin{equation}{section}
\renewcommand{\vline}{\mid}
\begin{document}
\title[Further Results On Uniqueness Of Derivatives Of Meromorphic Functions ]{Further Results On Uniqueness Of Derivatives Of Meromorphic Functions Sharing Three Sets  }
\numberwithin {equation}{section}
\date{}
\author[A.Banerjee, S.Majumder and B. Chakraborty]{Abhijit Banerjee$^{1}$, Sujoy Majumder$^{2}$ and Bikash Chakraborty$^{3}$}
\address{ $^{1}$  Department of Mathematics, University of Kalyani, West Bengal, India.}
\email{abanerjee\_kal@yahoo.co.in, abanerjee\_kal@rediffmail.com}
\address{$^{2}$ Department of Mathematics, Katwa College, Burdwan, India.}
\address{$^{2}$ Department of Mathematics, Raiganj University, Raiganj, India}
\email{sujoy.katwa@gmail.com, smajumder05@yahoo.in}
\address{ $^{3}$  Department of Mathematics, University of Kalyani, West Bengal, India.}
\address{ $^{3}$ Department of Mathematics, Ramakrishna Mission Vivekananda Centenary College, Rahara, India}
\email{bikashchakraborty.math@yahoo.com, bikashchakrabortyy@gmail.com}
\maketitle
\let\thefootnote\relax
\footnotetext{Type set by \AmS -\LaTeX}
\begin{abstract} We prove some uniqueness theorems concerning the derivatives of meromorphic functions when they share three sets which will improve some recent existing results.\end{abstract}
\section{Introduction, Definitions and Results}
In this paper by meromorphic functions we will always mean meromorphic functions in the complex plane.
We shall use the standard notations of value distribution theory :
$$T(r,f),\;\; m(r,f),\;\; N(r,\infty;f),\;\; \ol N(r,\infty;f),\ldots $$ (see \cite{4}).
It will be convenient to let $E$ denote any set of positive real numbers of finite linear measure, not necessarily the same at each occurrence. We denote by $T(r)$ the maximum of $T(r,f^{(k)})$ and $T(r,g^{(k)})$. The notation $S(r)$ denotes any quantity satisfying $S(r)=o(T(r))$ as $r\lra \infty, r\not\in E$.\par

If for some $a\in\mathbb{C}\cup\{\infty\}$, $f$ and $g$ have the same set of $a$-points with same multiplicities then we say that $f$ and $g$ share the value $a$ CM (counting multiplicities). If we do not take the multiplicities into account, $f$ and $g$ are said to share the value $a$ IM (ignoring multiplicities).

Let $S$ be a set of distinct elements of $\mathbb{C}\cup\{\infty\}$ and $E_{f}(S)=\bigcup_{a\in S}\{z: f(z)-a=0\}$, where each zero is counted according to its multiplicity. If we do not count the multiplicity the set $E_{f}(S)=\bigcup_{a\in S}\{z: f(z)-a=0\}$ is denoted by $\ol E_{f}(S)$.
If $E_{f}(S)=E_{g}(S)$ we say that $f$ and $g$ share the set $S$ CM. On the other hand if $\ol E_{f}(S)=\ol E_{g}(S)$, we say that $f$ and $g$ share the set $S$ IM. Evidently if $S$ contains only one element, then it coincides with the usual definition of CM(respectively, IM) shared values.\par
In 1926, R.Nevanlinna showed that a meromorphic function on the complex plane $\mathbb{C}$ is uniquely determined by the pre-images, ignoring multiplicities, of $5$ distinct values (including infinity). A few years latter, he showed that when multiplicities are taken into consideration, $4$ points are enough and in this case either the two functions coincides or one is the bilinear transformation of the other. \par
This two theories are the starting point of uniqueness theory. Research became more interesting although sophisticated when F.Gross and C.C.Yang transferred the study of uniqueness theory to a more general setup namely sets of distinct elements instead of values. For instance they proved that if $f$ and $g$ are two non-constant entire functions and $S_{1}$, $S_{2}$ and $S_{3}$ are three distinct finite sets such that $f^{-1}(S_{i})=g^{-1}(S_{i})$ for $i=1,2,3$ then $f\equiv g$.

The following analogous question corresponding to meromorphic functions was asked in \cite{20}.\\
{\bf Question A} {\it Can one find three finite sets $S_{j}$ $(j=1,2,3)$ such that any two non-constant meromorphic functions $f$ and $g$ satisfying $E_{f}(S_{j})=E_{g}(S_{j})$ for $j=1,2,3$  must be identical ?}

{\it Question A} may be considered as the inception of a new horizon in the uniqueness of meromorphic functions concerning three set sharing problem and so far the quest for affirmative answer to {\it Question A} under weaker hypothesis has made a great stride \{see \cite{1}-\cite{1a}, \cite{2b}-\cite{2c}, \cite{4}, \cite{13}, \cite{16}, \cite{18}-\cite{21}, \cite{25}\}.
But unfortunately the derivative counterparts of the above results are scanty in number.
In 2003, in the direction of {\it Question A} concerning the uniqueness of derivatives of meromorphic functions Qiu and Fang obtained the following result.
\begin{theoA}\cite{18} Let $S_{1}=\{z:z^{n}-z^{n-1}-1=0\}$, $S_{2}=\{\infty\}$ and $S_{3}=\{0\}$ and $n\;(\geq 3)$, $k$ be two positive integers. Let $f$ and $g$ be two non-constant meromorphic functions  such that $E_{f^{(k)}}(S_{j})=E_{g^{(k)}}(S_{j})$ for $j=1,3$ and $E_{f}(S_{2})=E_{g}(S_{2})$ then $f^{(k)}\equiv g^{(k)}$.\end{theoA}
In 2004 Yi and Lin \cite {25} independently proved the following theorem.
\begin{theoB}\cite{25}  Let $S_{1}=\{z:z^{n}+az^{n-1}+b=0\}$, $S_{2}=\{\infty\}$ and $S_{3}=\{0\}$, where $a$, $b$ are nonzero constants such that $z^{n}+az^{n-1}+b=0$ has no repeated root and $n\; (\geq 3)$, $k$ be two positive integers. Let $f$ and $g$ be two non-constant meromorphic functions  such that $E_{f^{(k)}}(S_{j})=E_{g^{(k)}}(S_{j})$ for $j=1,2,3$ then $f^{(k)}\equiv g^{(k)}$.\end{theoB}
The following examples show that in {\it Theorems A, B} $a\not=0$ is necessary.
\begin{exm}\label{e1.1}\cite{1aa} Let $f(z) =e^{z}$ and $g(z)=(-1)^{k}e^{-z}$ and $S_{1}=\{z:z^{3}-1=0\}$, $S_{2}=\{\infty\}$, $S_{3}=\{0\}$. Since $f^{(k)}-\omega ^{l}=g^{(k)}-\omega ^{3-l}$, where $\omega =cos\frac{2\pi}{3}+isin\frac{2\pi}{3}$, $0\leq l\leq 2$, clearly $E_{f^{(k)}}(S_{j})=E_{g^{(k)}}(S_{j})$ for $j=1,2,3$ but $f^{(k)}\not\equiv g^{(k)}$. \end{exm}
We now consider the following examples which establish the sharpness of the lower bound of $n$ in {\it Theorems A, B}.
\begin{exm}\label{e1.2}\cite{1aa} Let $f(z)=\sqrt {\alpha +\beta}\;\sqrt {\alpha \beta} \; e^{z}$ and $g(z)=(-1)^{k}\sqrt {\alpha +\beta}\;\sqrt {\alpha \beta}\; e^{-z}$ and $S_{1}=\{\alpha +\beta, \alpha\beta\}$, $S_{2}=\{\infty\}$, $S_{3}=\{0\}$, where $\alpha +\beta=-a$ and $\alpha \beta=b$; $a$, $b$ are nonzero complex numbers. Clearly $E_{f^{(k)}}(S_{j})=E_{g^{(k)}}(S_{j})$ for $j=1,2,3$ but $f^{(k)}\not\equiv g^{(k)}$. \end{exm}
\begin{exm}\label{e1.3} Let $f=\alpha \sqrt{\beta }e^{z}$, $g=(-1)^{k}\beta  \sqrt{\alpha  }e^{-z}$, where $\alpha$ and $\beta$ be two non zero complex numbers such that $\sqrt{\frac{\alpha }{\beta }}\not=-1$. Let $S_{1}=\{\beta  \sqrt{\alpha},\alpha \sqrt{\beta}\}$, $S_{2}=\{\infty\}$, $S_{3}=\{0\}$. Clearly $E_{f^{(k)}}(S_{j})=E_{g^{(k)}}(S_{j})$ for $j=1,2,3$ but $f^{(k)}\not\equiv g^{(k)}$.
\end{exm}
\begin{exm}\label{e1.4} Let $f=\sqrt{2}e^{z}$, $g=(-1)^{k}\sqrt{2}e^{-z}$. Let $S_{1}=\{1+i,1-i\}$, $S_{2}=\{\infty\}$, $S_{3}=\{0\}$. Clearly $E_{f^{(k)}}(S_{j})=E_{g^{(k)}}(S_{j})$ for $j=1,2,3$ but $f^{(k)}\not\equiv g^{(k)}$.
\end{exm}

Above example assures the fact that in {\it Theorems A}-{\it B}, the cardinality of the set $S_{1}$ can not be further reduced. Rather on the basis of above examples one may concentrate to relax the nature of sharing the range sets.
For the purpose of relaxation of the nature of sharing the sets the notion of weighted sharing of values and sets which appeared in \cite{10, 11} has become very much effective. We now give the definition.

\begin{defi}\label{d1.1} \cite{10, 11} Let $k$ be a nonnegative integer or infinity. For $a\in\mathbb{C}\cup\{\infty\}$ we denote by $E_{k}(a;f)$ the set of all $a$-points of $f$, where an $a$-point of multiplicity $m$ is counted $m$ times if $m\leq k$ and $k+1$ times if $m>k$. If $E_{k}(a;f)=E_{k}(a;g)$, we say that $f$, $g$ share the value $a$ with weight $k$.\end{defi}
The definition implies that if $f$, $g$ share a value $a$ with weight $k$ then $z_{0}$ is an $a$-point of $f$ with multiplicity $m\;(\leq k)$ if and only if it is an $a$-point of $g$ with multiplicity $m\;(\leq k)$ and $z_{0}$ is an $a$-point of $f$ with multiplicity $m\;(>k)$ if and only if it is an $a$-point of $g$ with multiplicity $n\;(>k)$, where $m$ is not necessarily equal to $n$.

We write $f$, $g$ share $(a,k)$ to mean that $f,g$ share the value $a$ with weight $k$. Clearly if $f$, $g$ share $(a,k)$ then $f$, $g$ share $(a,p)$ for any integer $p$, $0\leq p<k$. Also we note that $f$, $g$ share a value $a$ IM or CM if and only if $f$, $g$ share $(a,0)$ or $(a,\infty)$ respectively.
\begin{defi}\label{d1.2}\cite{10} Let $S$ be a set of distinct elements of $\mathbb{C}\cup\{\infty\}$ and $k$ be a nonnegative integer or $\infty$. We denote by $E_{f}(S,k)$ the set $E_{f}(S, k)=\bigcup_{a\in S}E_{k}(a;f)$.

Clearly $E_{f}(S)=E_{f}(S,\infty)$ and $\ol E_{f}(S)=E_{f}(S,0)$.\end{defi}
In 2009 Banerjee and Bhattacharjee \cite{1b} subtly use the concept of weighted sharing of sets to improve {\it Theorems A} and {\it B} as follows :
\begin{theoC}\cite{1b} Let $S_{i}$, $i=1,2,3$ be defined as in {\em Theorem B} and $k$ be a positive integer. If $f$ and $g$ are two non-constant meromorphic functions such that $E_{f^{(k)}}(S_{1},4)=E_{g^{(k)}}(S_{1},4)$, $E_{f}(S_{2},\infty)=E_{g}(S_{2},\infty)$ and $E_{f^{(k)}}(S_{3},7)=E_{g^{(k)}}(S_{3},7)$ then $f^{(k)}\equiv g^{(k)}$. \end{theoC}
\begin{theoD}\cite{1b} Let $S_{i}$, $i=1,2,3$ be defined as in {\em Theorem B} and $k$ be a positive integer. If $f$ and $g$ are two non-constant meromorphic functions such that $E_{f^{(k)}}(S_{1},5)=E_{g^{(k)}}(S_{1},5)$, $E_{f}(S_{2},\infty)=E_{g}(S_{2},\infty)$ and $E_{f^{(k)}}(S_{3},1)=E_{g^{(k)}}(S_{3},1)$ then $f^{(k)}\equiv g^{(k)}$. \end{theoD}
\begin{theoE}\cite{1b} Let $S_{i}$, $i=1,2,3$ be defined as in {\em Theorem B} and $k$ be a positive integer. If $f$ and $g$ are two non-constant meromorphic functions such that $E_{f^{(k)}}(S_{1},6)=E_{g^{(k)}}(S_{1},6)$, $E_{f}(S_{2},\infty)=E_{g}(S_{2},\infty)$ and $E_{f^{(k)}}(S_{3},0)=E_{g^{(k)}}(S_{3},0)$ then $f^{(k)}\equiv g^{(k)}$. \end{theoE}

A few years latter in 2011 Banerjee and Bhattacharjee \cite{1aa} further improved the above results in the following manner.
\begin{theoF}\cite{1aa} Let $S_{i}$, $i=1,2,3$ be defined as in {\em Theorem B} and $k$ be a positive integer. If $f$ and $g$ are two non-constant meromorphic functions such that $E_{f^{(k)}}(S_{1},5)=E_{g^{(k)}}(S_{1},5)$, $E_{f}(S_{2},\infty)=E_{g}(S_{2},\infty)$ and $E_{f^{(k)}}(S_{3},0)=E_{g^{(k)}}(S_{3},0)$ then $f^{(k)}\equiv g^{(k)}$. \end{theoF}
\begin{theoG}\cite{1aa} Let $S_{i}$, $i=1,2,3$ be defined as in {\em Theorem B} and $k$ be a positive integer. If $f$ and $g$ are two non-constant meromorphic functions such that $E_{f^{(k)}}(S_{1},4)=E_{g^{(k)}}(S_{1},4)$, $E_{f}(S_{2},\infty)=E_{g}(S_{2},\infty)$ and $E_{f^{(k)}}(S_{3},1)=E_{g^{(k)}}(S_{3},1)$ then $f^{(k)}\equiv g^{(k)}$. \end{theoG}
\begin{theoH}\cite{1aa} Let $S_{i}$, $i=1,2,3$ be defined as in {\em Theorem B} and $k$ be a positive integer. If $f$ and $g$ are two non-constant meromorphic functions such that $E_{f^{(k)}}(S_{1},5)=E_{g^{(k)}}(S_{1},5)$, $E_{f}(S_{2},9)=E_{g}(S_{2},9)$ and $E_{f^{(k)}}(S_{3},\infty)=E_{g^{(k)}}(S_{3},\infty)$ then $f^{(k)}\equiv g^{(k)}$. \end{theoH}
In the present paper we we significantly reduce the weight of the range sets in all the above theorems.  The following theorems are the main results of the paper:
\begin{theo}\label{t1.1} Let $S_{i}$, $i=1,2,3$ be defined as in {\em Theorem B} and $k$ be a positive integer. If $f$ and $g$ are two non-constant meromorphic functions such that $E_{f^{(k)}}(S_{1},k_{1})=E_{g^{(k)}}(S_{1},k_{1})$, $E_{f}(S_{2},k_{2})=E_{g}(S_{2},k_{2})$ and $E_{f^{(k)}}(S_{3},k_{3})=E_{g^{(k)}}(S_{3},k_{3})$, where $k_{1}\geq 4$, $k_{2}\geq 0$, $k_{3}\geq 0$ are integers satisfying $$2k_{1}k_{2}k_{3}>k_{1}+k_{2}+2k_{3}+k-2kk_{1}k_{3}-k_{1}k_{2}-kk_{1}+3,$$ then $f^{(k)}\equiv g^{(k)}$. \end{theo}
\begin{rem} Note that {\em Theorem 1.1} holds for $k_{1}= 4$, $k_{2}= 2$ and $k_{3}= 0$. So {\em Theorem 1.1} improves {\em Theorems A-H}.\end{rem}
\begin{rem} {\em Examples 1.2-1.4} assures the fact that in {\em Theorem \ref{t1.1}}, $n\geq 3$ is the best possible.\end{rem}
Though we follow the standard definitions and notations of the value distribution theory available in \cite{7}, we explain some notations which are used in the paper.
\begin{defi} \label{d1.3}\cite{9} For $a\in\mathbb{C}\cup\{\infty\}$we denote by $N(r,a;f\vline=1)$ the counting function of simple $a$ points of $f$. For a positive integer $m$ we denote by $N(r,a;f\vline\leq m) (N(r,a;f\vline\geq m))$ the counting function of those $a$ points of $f$ whose multiplicities are not greater(less) than $m$ where each $a$ point is counted according to its multiplicity.

$\ol N(r,a;f\vline\leq m)\; (\ol N(r,a;f\vline\geq m))$ are defined similarly, where in counting the $a$-points of $f$ we ignore the multiplicities.

Also $N(r,a;f\vline <m),\; N(r,a;f\vline >m),\; \ol N(r,a;f\vline <m)\; and\; \ol N(r,a;f\vline >m)$ are defined analogously.  \end{defi}
\begin{defi}\label{d1.4} We denote by $\ol N(r,a;f\vline=k)$ the reduced counting function of those $a$-points of $f$ whose multiplicities is exactly $k$, where $k\geq 2$ is an integer. \end{defi}
\begin{defi} \label{d1.5}\cite{1a} Let $f$ and $g$ be two non-constant meromorphic functions such that $f$ and $g$ share $(a,k)$ where $a\in\mathbb{C}\cup\{\infty\}$. Let $z_{0}$ be an $a$-point of $f$ with multiplicity $p$, a $a$-point of $g$ with multiplicity $q$. We denote by $\ol N_{L}(r,a;f)$ the counting function of those $a$-points of $f$ and $g$ where $p>q$; each point in this counting functions is counted only once. In the same way we can define $\ol N_{L}(r,a;g)$. \end{defi}
\begin{defi}\label{d1.6} \cite{11} We denote $N_{2}(r,a;f)=\ol N(r,a;f)+\ol N(r,a;f\vline\geq 2)$\end{defi}
\begin{defi}\label{d1.7}\cite{10, 11} Let $f$, $g$ share a value $a$ IM. We denote by $\ol N_{*}(r,a;f,g)$ the reduced counting function of those $a$-points of $f$ whose multiplicities differ from the multiplicities of the corresponding $a$-points of $g$.

Clearly $\ol N_{*}(r,a;f,g)\equiv\ol N_{*}(r,a;g,f)$ and $\ol N_{*}(r,a;f,g)=\ol N_{L}(r,a;f)+\ol N_{L}(r,a;g)$.\end{defi}

\begin{defi}\label{d1.8}\cite{14} Let $a,b \in\mathbb{C}\cup\{\infty\}$. We denote by $N(r,a;f\vline\; g=b)$ the counting function of those $a$-points of $f$, counted according to multiplicity, which are $b$-points of $g$.\end{defi}
\begin{defi}\label{d1.9}\cite{14} Let $a,b_{1},b_{2},\ldots,b_{q} \in\mathbb{C}\;\cup\{\infty\}$. We denote by $N(r,a;f\vline\; g\neq b_{1},b_{2},\ldots,b_{q})$ the counting function of those $a$-points of $f$, counted according to multiplicity, which are not the $b_{i}$-points of $g$ for $i=1,2,\ldots,q$. \end{defi}
\begin{defi}\label{d1.10} Let $f$ and $g$ be two non-constant meromorphic functions such that $E_{f}(S,k)=E_{g}(S,k)$. Let $a$ and $b$ be any two elements of $S$. We denote by $\ol N_{*}(r,a;f|g=b)$ the reduced counting function of those $a$-points of $f$ whose multiplicities differ from the multiplicities of the corresponding $b$-points of $g$.\\ Clearly $\ol N_{*}(r,a;f|g=b)=\ol N_{*}(r,b;g|f=a)$. Also if $a=b$, then $\ol N_{*}(r,a;f|g=b)=\ol N_{*}(r,a;f,g)$.   \end{defi}

\section {Lemmas} In this section we present some lemmas which will be needed in the sequel. Let $F$ and $G$ be two non-constant meromorphic functions defined as follows.\\
\be{\label{e2.1}}F=\frac{\left(f^{(k)}\right)^{n-1}(f^{(k)}+a)}{-b},\;\;\;\;\;\;\;\; G=\frac{\left(g^{(k)}\right)^{n-1}(g^{(k)}+a)}{-b},\ee where $n(\geq 2)$ and $k$ are two positive integers.

Henceforth we shall denote by $H$, $\Phi_{1}$,$\Phi_{2}$ and $\Phi_{3}$ the following three functions $$H=(\frac{F^{''}}{F^{'}}-\frac{2F^{'}}{F-1})-(\frac{G^{''}}{G^{'}}-\frac{2G^{'}}{G-1}),$$ $$\Phi_{1}=\frac{F^{'}}{F-1}-\frac{G^{'}}{G-1},$$ $$\Phi_{2}=\frac{(f^{(k)})^{'}}{f^{(k)}}-\frac{(g^{(k)})^{'}}{g^{(k)}} $$ and $$\Phi_{3}=(\frac{(f^{(k)})^{'}}{f^{(k)}-\omega_{i}}-\frac{(f^{(k)}){'}}{f^{(k)}})-(\frac{(g^{(k)})^{'}}{g^{(k)}-\omega_{j}}-\frac {(g^{(k)}){'}}{g^{(k)}}),$$ where $\omega_{i}$ and $\omega_{j}$  be any two roots of the equation $z^{n}+az^{n-1}+b=0$
\begin{lem}\label{l2.1}(\cite{11}, Lemma 1) Let $F$, $G$ share $(1,1)$ and $H\not\equiv 0$. Then $$N(r,1;F\mid=1)=N(r,1;G\mid=1)\leq N(r,H)+S(r,F)+S(r,G).$$\end{lem}
\begin{lem}\label{l2.2} Let $S_{1}$, $S_{2}$ and $S_{3}$ be defined as in {\em Theorem \ref{t1.1}} and $F$, $G$ be given by (\ref{e2.1}). If for two non-constant meromorphic functions $f$ and $g$ $E_{f^{(k)}}(S_{1},0)=E_{g^{(k)}}(S_{1},0)$, $E_{f^{(k)}}(S_{2},0)=E_{g^{(k)}}(S_{2},0)$, $E_{f}(S_{3},0)=E_{g}(S_{3},0)$ and $H\not\equiv 0$ then \beas N(r,H)&\leq&\ol N_{*}(r,0;f^{(k)},g^{(k)})+\ol N_{*}(r,1;F,G)+\ol N(r,-a\frac{n-1}{n};f^{(k)})+\ol N(r,-a\frac{n-1}{n};g^{(k)})\\& &+\ol N_{*}(r,\infty;f,g)+\ol N_{0}(r,0;(f^{(k)})^{'})+\ol N_{0}(r,0;(g^{(k)})^{'}),\eeas where $\ol N_{0}(r,0;(f^{(k)})^{'})$ is the reduced counting function of those zeros of $(f^{(k)})^{'}$ which are not the zeros of $f^{(k)}(F-1)$ and $\ol N_{0}(r,0;(g^{(k)})^{'})$ is similarly defined.\end{lem}
\begin{proof} Since $E_{f^{(k)}}(S_{1},0)=E_{g^{(k)}}(S_{1},0)$ it follows that $F$ and $G$ share $(1,0)$.We have from (\ref{e2.1}) that $$F^{'}=[nf^{(k)}+(n-1)a](f^{(k)})^{n-2}(f^{(k)})^{'}/(-b)$$ and $$G^{'}=[ng^{(k)}+(n-1)a](g^{(k)})^{n-2}(g^{(k)})^{'}/(-b).$$ We can easily verify that possible poles of $H$ occur at (i) those zeros of $f^{(k)}$ and $g^{(k)}$ whose multiplicities are distinct from the multiplicities of the corresponding zeros of $g^{(k)}$ and $f^{(k)}$ respectively, (ii)zeros of $nf^{(k)}+a(n-1)$ and $ng^{(k)}+a(n-1)$, (iii) those poles of $f$ and $g$ whose multiplicities are distinct from the multiplicities of the corresponding poles of $g$ and $f$ respectively, (iv) those $1$-points of $F$ and $G$ with different multiplicities, (v) zeros of $(f^{(k)})^{'}$ which are not the zeros of $f^{(k)}(F-1)$, (vi) zeros of $(g^{(k)})^{'}$ which are not zeros of $g^{(k)}(G-1)$.\end{proof}
\begin{lem}\label{l2.4}\cite{17e} Let $f$ be a non-constant meromorphic function and let \[R(f)=\frac{\sum\limits _{k=0}^{n} a_{k}f^{k}}{\sum \limits_{j=0}^{m} b_{j}f^{j}}\] be an irreducible rational function in $f$ with constant coefficients $\{a_{k}\}$ and $\{b_{j}\}$where $a_{n}\not=0$ and $b_{m}\not=0$ Then $$T(r,R(f))=dT(r,f)+S(r,f),$$ where $d=\max\{n,m\}$.\end{lem}
\begin{lem}\label{l2.5}\cite{1aa} Let $F$ and $G$ be given by (\ref{e2.1}). If  $f^{(k)}$, $g^{(k)}$ share $(0,0)$ and $0$ is not a Picard exceptional value of $f^{(k)}$ and $g^{(k)}$. Then $\Phi_{1} \equiv 0$  implies $F\equiv G$.\end{lem}
\begin{lem}\label{l2.6}\cite{1aa} Let $F$ and $G$ be given by (\ref{e2.1}), $n\geq 3$ be an integer and $\Phi_{1} \not\equiv 0$. If $F$, $G$ share $(1,k_{1})$; $f$, $g$ share $(\infty,k_{2})$, and $f^{(k)}$, $g^{(k)}$ share $(0,k_{3})$, where $0\leq k_{3}< \infty$ then \beas [(n-1)k_{3}+n-2]\;\ol N(r,0;f^{(k)}\mid\geq k_{3}+1)&\leq&\;\ol N_{*}(r,1;F,G)+\ol N_{*}(r,\infty;F,G)+S(r).\eeas \end{lem}

\begin{lem}\label{l2.8} Let $f$, $g$ be two non-constant meromorphic functions, $F$, $G$ be given by (\ref{e2.1}), $n\geq 3$ be an integer and $\Phi_{2}\not\equiv 0$. If $F$, $G$ share $(1,k_{1})$ and $f^{(k)}$, $g^{(k)}$ share $(0,k_{3})$; $f$, $g$ share $(\infty,k_{2})$, where $1\leq k_{1}\leq \infty$ then
\beas k_{1}\ol N(r,1;F|\geq k_{1}+1) \leq \ol N_{*}(r,0;f^{(k)},g^{(k)})+\ol N_{*}(r,\infty;f,g)+S(r,f^{(k)})+S(r,g^{(k)}).\eeas \end{lem}
\begin{proof} Note that \beas& & k_{1}\;\ol N(r,1;F|\geq k_{1}+1)\\&\leq& N(r,0;\Phi_{2})\\&\leq& N(r,\Phi_{2})+S(r,f^{(k)})+S(r,g^{(k)})\\&\leq&  \ol N_{*}(r,0;f^{(k)},g^{(k)})+\ol N_{*}(r,\infty;f,g)+S(r,f^{(k)})+S(r,g^{(k)}).\eeas  \end{proof}
\begin{lem}\label{l2.9} Let $f$, $g$ be two non-constant meromorphic functions. Also let $F$, $G$ be given by (\ref{e2.1}), $n\geq 3$ an integer and $\Phi_{1}\not\equiv 0$,$\Phi_{2}\not\equiv 0$. If $F$, $G$ share $(1,k_{1})$, where $k_{1}\geq 2$, $f^{(k)}$, $g^{(k)}$ share $(0,k_{3})$ and $f$,  $g$ share $(\infty,k_{2})$, $0\leq k_{2}\leq \infty$ then
\beas \ol N(r,0;f^{(k)}|\geq k_{3}+1) \leq \frac{k_{1}+1}{k_{1}[(n-1)k_{3}+(n-2)]-1}\ol N_{*}(r,\infty;f,g)+S(r,f^{(k)})+S(r,g^{(k)}),\eeas Similar result holds for $g^{(k)}$.\end{lem}
\begin {proof} Using {\it Lemma \ref{l2.6}} and {\it Lemma \ref{l2.8}} and noting that \\ $\ol N_{*}(r,0;f^{(k)},g^{(k)})\leq \ol N(r,0;f^{(k)}|\geq k_{3}+1)$ we see that \beas & &[(n-1)k_{3}+(n-2)]\ol N(r,0;f^{(k)}|\geq k_{3}+1)\\&\leq&\ol N_{*}(r,1;F,G)+\ol N_{*}(r,\infty;f,g)+S(r,f^{(k)})+S(r,g^{(k)})\\&\leq& \frac{1}{k_{1}}\ol N(r,0;f^{(k)}|\geq k_{3}+1)+\frac{k_{1}+1}{k_{1}}\ol N_{*}(r,\infty;f,g)\nonumber\\ & &+S(r,f^{(k)})+S(r,g^{(k)}),\eeas
from which the lemma follows.\end{proof}
\begin{lem}\label{l2.9a} Let $f$ and $g$ be two non-constant meromorphic functions. Suppose $f$, $g$ share $(\infty,0)$ and $\infty$ is not an Picard exceptional value of $f$ and $g$.  Then $\Phi_{3}\equiv 0$  implies $f^{(k)}\equiv \frac{\omega_{i}}{\omega_{j}}g^{(k)}$.\end{lem}
\begin{proof} Suppose $\Phi_{3}\equiv 0$. Then by integration we obtain $$1-\frac{\omega_{i}}{f^{(k)}}\equiv A(1-\frac{\omega_{j}}{g^{(k)}}),$$ where $A\not=0$. Since $f$, $g$ share $(\infty,0)$ it follows that $A=1$ and hence $f^{(k)}\equiv \frac{\omega_{i}}{\omega_{j}}g^{(k)}$. \end {proof}
\begin{lem}\label{l2.10} Let $f$ and $g$ be two non-constant meromorphic functions and $\Phi_{3}\not\equiv 0$. Also let $F$ and $G$ be given by (\ref{e2.1}). If $f^{(k)}$, $g^{(k)}$ share $(0,k_{3})$; $f$ and $g$ share $(\infty,k_{2})$, where $1\leq k_{2}\leq \infty$ and $E_{f^{(k)}}(S_{1},k_{1})=E_{g^{(k)}}(S_{1},k_{1})$, where $S_{1}$ is the same set as used in the {\it Theorem 1.1} and $1\leq k_{1}\leq \infty$ then \beas & & (k_{2}+k)\;\;\ol N(r,\infty;f\mid\geq k_{2}+1)\\&\leq& \ol N_{*}\left(r,0;f^{(k)},g^{(k)}\right)+\ol N_{*}(r,1;F,G)+S(r).\eeas
Similar expressions hold for $g^{(k)}$ also. \end{lem}
\begin{proof} If $\infty$ is an e.v.P of $f^{(k)}$ and $g^{(k)}$ then the assertion follows immediately.\\
Next suppose $\infty$ is not an e.v.P of $f^{(k)}$ and $g^{(k)}$. Since $E_{f^{(k)}}(S_{1},k_{1})=E_{g^{(k)}}(S_{1},k_{1})$, it follows that $\ol N_{*}(r,\omega_{i};f^{(k)}|g^{(k)}=\omega_{j})\leq \ol N_{*}(r,1;F,G)$. Note that \beas& & (k_{2}+k)\;\;\ol N(r,\infty;f|\geq k_{2}+1)\\&=& (k_{2}+k)\;\ol N(r,\infty;g|\geq k_{2}+1)\\&\leq& N(r,0;\Phi_{3})\\&\leq& N(r,\Phi_{3})+S(r,f^{(k)})+S(r,g^{(k)})\\&\leq&  \ol N_{*}(r,0;f^{(k)},g^{(k)})+\ol N_{*}(r,\omega_{i};f^{(k)}|g^{(k)}=\omega_{j})+S(r,f^{(k)})+S(r,g^{(k)})\\&\leq&  \ol N_{*}(r,0;f^{(k)},g^{(k)})+\ol N_{*}(r,1;F,G)+S(r).\eeas \end{proof}
\begin{lem}\label{l2.11} Let $f$, $g$ be two non-constant meromorphic functions and $\Phi_{2}\not\equiv 0$,  $\Phi_{3}\not\equiv 0$. Also let $F$ and $G$ be given by (\ref{e2.1}).If $f^{(k)}$, $g^{(k)}$ share $(0,k_{3})$; $f$ and $g$ share $(\infty,k_{2})$, where $0\leq k_{2}<\infty$ and $F$, $G$ share $(1,k_{1})$, where $k_{1}> 1$ then \beas & &  \;\ol N(r,\infty;f\mid\geq k_{2}+1)\\&\leq& \frac{k_{1}+1}{k_{1}(k_{2}+k)-1}\ol N_{*}\left(r,0;f^{(k)},g^{(k)}\right)+S(r).\eeas
Similar expressions hold for $g^{(k)}$ also. \end{lem}
\begin {proof} Using {\it Lemma \ref{l2.8}} and {\it Lemma \ref{l2.10}} and noting that $\ol N_{*}(r,\infty;f,g)\leq \ol N(r,\infty;f|\geq k_{2}+1)$ we see that \beas (k_{2}+k)\ol N(r,\infty;f|\geq k_{2}+1)&\leq&\ol N_{*}(r,1;F,G)+\ol N_{*}(r,0;f^{(k)},g^{(k)})+S(r)\\&\leq& \frac{1}{k_{1}}\ol N(r,\infty;f|\geq k_{2}+1)+\frac{k_{1}+1}{k_{1}}\ol N_{*}(r,0;f^{(k)},g^{(k)})\\ & &+S(r,f^{(k)})+S(r,g^{(k)}),\eeas
from which the lemma follows.\end{proof}
\begin{lem}\label{l2.12aa}\cite{1aa}  Let $F$, $G$ be given by (\ref{e2.1}) and $H\not\equiv 0$. If $f^{(k)}$, $g^{(k)}$ share $(0,k_{3})$; $f$ and $g$ share $(\infty,k_{2})$, where $0\leq k_{2}<\infty$ and $F$, $G$ share $(1,k_{1})$, where $1\leq k_{1}\leq \infty$ then \beas & & \{(nk_{2}+nk+n)-1\} \;\ol N(r,\infty;f\mid\geq k_{2}+1)\\&\leq& \ol N_{*}\left(r,0;f^{(k)},g^{(k)}\right)+\ol N\left(r,0;f^{(k)}+a\right)+\ol N\left(r,0;g^{(k)}+a\right)+\ol N_{*}(r,1;F,G)+S(r).\eeas
Similar expressions hold for $g$ also. \end{lem}
\begin{lem}\label{l2.14} Let $f$ ,$g$ be two non-constant meromorphic functions. Also let $F$, $G$ be given by (\ref{e2.1}),$n\geq 3$ an integer and $\Phi_{1}\not\equiv 0$, $\Phi_{2}\not\equiv 0$ and $\Phi_{3}\not\equiv 0$. If $F$, $G$ share $(1,k_{1})$; $f^{(k)}$,  $g^{(k)}$ share $(0,k_{3})$ and $f$, $g$ share $(\infty,k_{2})$, where $k_{1}>1$, $k_{2}\geq 0$ and $k_{3}\geq 0$ are integers satisfying $$ 2k_{1}k_{2}k_{3}>k_{1}+k_{2}+2k_{3}+k-2kk_{1}k_{3}-k_{1}k_{2}-kk_{1}+3,$$ then \beas \ol N(r,1;F|\geq k_{1}+1)+\ol N(r,\infty;f|\geq k_{2}+1)+\ol N(r,0;f^{(k)}|\geq k_{3}+1)=S(r).\eeas \end{lem}
\begin {proof} Since $\Phi_{1}\not\equiv 0$ we get from the {\it Lemma \ref{l2.6}} we get \beas (2k_{3}+1)\;\ol N(r,0;f^{(k)}|\geq k_{3}+1)&\leq&\ol N(r,1;F\mid\geq k_{1}+1)+\ol N(r,\infty;f|\geq k_{2}+1)+S(r).\eeas
Again since $\Phi_{2}\not\equiv 0$ and $\Phi_{3}\not\equiv 0$ we get by {\it Lemmas \ref{l2.8}}, {\ref{l2.10}} respectively \beas k_{1}\;\ol N(r,1;F|\geq k_{1}+1)&\leq&\ol N(r,0;f^{(k)}\mid\geq k_{3}+1)+\ol N(r,\infty;f|\geq k_{2}+1)+S(r),\eeas
and \beas (k_{2}+k)\;\;\ol N(r,\infty;f|\geq k_{2}+1)&\leq&\ol N(r,1;F\mid\geq k_{1}+1)+\ol N(r,0;f^{(k)}|\geq k_{3}+1)+S(r).\eeas

Using the above inequalities and following the same procedure as done in {\it Lemma 2.6} \cite{20a} the rest of the lemma can be proved. So we omit the details.\end{proof}
\begin{lem}\label{l2.15}\cite{11} If $N(r,0;f^{(k)}\vline f\not=0)$ denotes the counting function of those zeros of $f^{(k)}$ which are not the zeros of $f$, where a zero of $f^{(k)}$ is counted according to its multiplicity then $$N(r,0;f^{(k)}\vline f\not=0)\leq k\ol N(r,\infty;f)+N(r,0;f\vline <k)+k\ol N(r,0;f\vline\geq k)+S(r,f).$$\end{lem}

\begin{lem}\label{l2.16} Let $F$, $G$ be given by (\ref{e2.1}), $F$, $G$ share $(1,k_{1})$, $2\leq k_{1}\leq \infty$ and $\Phi_{1}\not\equiv 0$  and $n\geq 3$. Also $f^{(k)}$,  $g^{(k)}$ share $(0,k_{3})$ and $f$, $g$ share $(\infty,\infty)$. Then  \beas \ol N(r,0;f^{(k)})\leq \frac{1}{k_{1}(n-2)-1}\;\;\ol N(r,\infty;f)+S(r,f^{(k)}).\eeas \end{lem}
\begin {proof} Using {\it Lemma \ref{l2.4}} and {\it Lemma \ref{l2.15}} we see that \beas \ol N_{*}(r,1;F,G)&\leq&\ol N(r,1;F\mid\geq k_{1}+1)\\&\leq& \frac{1}{k_{1}}\;\left(N(r,1:F)-\ol N(r,1;F)\right)\\&\leq &\frac{1}{k_{1}}\;[\sum_{j=1}^{n} \left( N(r,\omega _{j};f^{(k)})-\ol N(r,\omega_{j};f^{(k)})\right)]\\&\leq& \frac{1}{k_{1}}\;\left(N(r,0;(f^{(k)})^{'}\mid f^{(k)} \not=0)\right)\\&\leq&\frac{1}{k_{1}}\left[\ol N(r,0;f^{(k)})+\ol N(r,\infty;f)\right]+S(r,f^{(k)}),\eeas where $\omega_{1},\omega_{2}\ldots \omega_{n}$ are the distinct roots of the equation $z^{n}+az^{n-1}+b=0$.
Rest of the proof follows from the {\it Lemma  \ref{l2.6}} for $k_{3}=0$.This proves the lemma.\end{proof}
\begin{lem}\label{l2.17} Let $F$, $G$ be given by (\ref{e2.1}), $F$, $G$ share $(1,k_{1})$, $2\leq k_{1}\leq \infty$ and $\Phi_{1}\not\equiv 0$  and $n\geq 3$. Also $f^{(k)}$,  $g^{(k)}$ share $(0,k_{3})$ and $f$,   $g$ share $(\infty,\infty)$, where $0\leq k_{3}\leq \infty$. Then  \beas \ol N_{L}(r,1;F)\leq \frac{k_{1}(n-2)}{(k_{1}+1)[k_{1}(n-2)-1]}\;\;\ol N(r,\infty;f)+S(r,f^{(k)}).\eeas Similar expression holds for $G$ also. \end{lem}
\begin {proof} Using {\it Lemma \ref{l2.4}} and {\it Lemma \ref{l2.15}} we see that \beas \ol N_{L}(r,1;F)&\leq&\ol N(r,1;F\mid\geq k_{1}+2)\\&\leq& \frac{1}{k_{1}+1}\;\left(N(r,1:F)-\ol N(r,1;F)\right)\\&\leq&\frac{1}{k_{1}+1}\left[\ol N(r,0;f^{(k)})+\ol N(r,\infty;f)\right]+S(r,f^{(k)}).\eeas
Now using {\it Lemma \ref{l2.16}} the rest of the lemma can be easily proved. So we omit it.\end{proof}
\begin{lem}\label{l2.18}\cite{1} Let $f$ and $g$ be two non-constant meromorphic functions sharing $(1,k_{1})$, where $2\leq k_{1}\leq \infty$. Then\beas \ol N(r,1;f|= 2)+2\;\ol N(r,1;f|=3)+\ldots+(k_{1}-1)\;\ol N(r,1;f|=k_{1})+k_{1}\;\ol N_{L}(r,1;f)\\+(k_{1}+1)\;\ol N_{L}(r,1;g)+k_{1}\;\ol N_{E}^{(k_{1}+1}(r,1;g)\leq N(r,1;g)-\ol N(r,1;g).\eeas \end{lem}
\begin{lem}\label{l2.19} Let $F$, $G$ be given by (\ref{e2.1}) and they share $(1,k_{1})$. If $f^{(k)}$, $g^{(k)}$ share $(0,k_{3})$ and $f$,  $g$ share  $(\infty,k_{2})$, where $2\leq k_{1}\leq \infty$ and $H\not\equiv 0$.
\beas  nT(r,f^{(k)})&\leq& \ol N(r,\infty;f)+\ol N(r,-a\frac{n-1}{n};f^{(k)})+\ol N(r,\infty;g)+\ol N(r,-a\frac{n-1}{n};g^{(k)})\\ & &+\ol N(r,0;f^{(k)})+\ol N(r,0;g^{(k)})+\ol N_{*}(r,0;f^{(k)},g^{(k)})+\ol N_{*}(r,\infty;f,g)\nonumber\\ & &-(k_{1}-1)\ol N_{*}(r,1;F,G)+\ol N_{L}(r,1;F)+S(r,f^{(k)})+S(r,g^{(k)}).\eeas Similar result holds for $g^{(k)}$.\end{lem}
\begin{proof} Using {\it Lemma  \ref{l2.15}} and {\it Lemma \ref{l2.18}} we see that
\bea \label{e2.2}& & \ol N_{0}(r,0;(g^{(k)})^{'})+\ol N(r,1;F\mid\geq 2)+\ol N_{*}(r,1;F,G)\\&\leq& \ol N_{0}(r,0;(g^{(k)})^{'})+\ol N(r,1;F|=2)+\ol N(r,1;F|=3)+\ldots +\ol N(r,1;F|=k_{1})\nonumber\\ & &+\ol N_{E}^{(k_{1}+1}(r,1;F)+\ol N_{L}(r,1;F)+\ol N_{L}(r,1;G)+\ol N_{*}(r,1;F,G)\nonumber\\&\leq& \ol N_{0}(r,0;(g^{(k)})^{'})-\ol N(r,1;F|=3)-\ldots -(k_{1}-2)\ol N(r,1;F|=k_{1})-(k_{1}-1)\ol N_{L}(r,1;F)\nonumber\\ & &-k_{1}\ol N_{L}(r,1;G)-(k_{1}-1)\ol N_{E}^{(k_{1}+1}(r,1;F)+N(r,1;G)-\ol N(r,1;G)+\ol N_{*}(r,1;F,G)\nonumber\\&\leq& \ol N_{0}(r,0;(g^{(k)})^{'})+N(r,1;G)-\ol N(r,1;G)-(k_{1}-2)\ol N_{L}(r,1;F)-(k_{1}-1)\ol N_{L}(r,1;G)\nonumber\\&\leq&\ol N_{0}(r,0;(g^{(k)})^{'})+N(r,1;G)-\ol N(r,1;G)-(k_{1}-2)\ol N_{L}(r,1;F)-(k_{1}-1)\ol N_{L}(r,1;G)\nonumber\\&\leq& N(r,0;(g^{(k)})^{'}\mid g^{(k)}\not=0)-(k_{1}-2)\ol N_{L}(r,1;F)-(k_{1}-1)\ol N_{L}(r,1;G)\nonumber\\&\leq& \ol N(r,0;g^{(k)})+\ol N(r,\infty;g)-(k_{1}-2)\ol N_{L}(r,1;F)-(k_{1}-1)\ol N_{L}(r,1;G)\nonumber\\&=& \ol N(r,0;g^{(k)})+\ol N(r,\infty;g)-(k_{1}-1)\ol N_{*}(r,1;F,G)+\ol N_{L}(r,1;F),\nonumber \eea
where $\ol N_{0}(r,0;(g^{(k)})^{'})$ has the same meaning as in the {\it Lemma  \ref{l2.2}}. Hence using (2.2), {\it Lemmas \ref{l2.1}}, {\it \ref{l2.2}} and {\it \ref{l2.4}} we get from second fundamental theorem that
\bea \label{e2.3}& & n\;T(r,f^{(k)})\\ &\leq & \ol N(r,0;f^{(k)})+\ol N(r,\infty;f)+ N(r,1;F\mid=1)+\ol N(r,1;F\mid\geq 2)-N_{0}(r,0;(f^{(k)})^{'})\nonumber\\&&+S(r,f^{(k)})\nonumber\\&\leq & \ol N(r,0;f^{(k)})+\ol N(r,\infty;f)+\ol N(r,-a\frac{n-1}{n};f^{(k)})+\ol N(r,-a\frac{n-1}{n};g^{(k)})\nonumber\\ & &+\ol N_{*}(r,0;f^{(k)},g^{(k)})+\ol N_{*}(r,\infty;f,g)+ \ol N_{*}(r,1;F,G)+\ol N(r,1;F|\geq 2)+\ol N_{0}(r,0;(g^{(k)})^{'})\nonumber\\ & &+S(r,f^{(k)})+S(r,g^{(k)})\nonumber \\ &\leq & \ol N(r,\infty;f)+\ol N(r,-a\frac{n-1}{n};f^{(k)})+\ol N(r,\infty;g)+\ol N(r,-a\frac{n-1}{n};g^{(k)})+\ol N(r,0;f^{(k)})\nonumber\\& & +\ol N(r,0;g^{(k)})+\ol N_{*}(r,\infty;f,g)+\ol N_{*}(r,0;f^{(k)},g^{(k)})-(k_{1}-1)\ol N_{*}(r,1;F,G)+\ol N_{L}(r,1;F)\nonumber\\ & &+S(r,f^{(k)})+S(r,g^{(k)}).\nonumber\eea\par
This proves the Lemma.\end{proof}
\begin{lem}\label{l2.24}\cite{1aa} Let $F$, $G$ be given by (\ref{e2.1}), $n\geq 3$ and they share $(1,k_{1})$. If $f^{(k)}$, $g^{(k)}$ share $(0,0)$, and $f$, $g$ share $(\infty,k_{2})$ and $H\equiv 0$. Then $f^{(k)}\equiv g^{(k)}$. \end{lem}
\section {Proofs of the theorem}
\begin{proof} [Proof of Theorem \ref{t1.1}]  Let $F$, $G$ be given by (\ref{e2.1}). Then $F$ and $G$ share $(1,k_{1})$, $(\infty;k_{2})$. We consider the following cases.\\
{\bf Case 1.} Let $H\not\equiv 0$. Clearly $F\not\equiv G$ and so $f^{(k)}\not\equiv g^{(k)}$.\\
{\bf Subcase 1.1:} Let $\Phi_{2}\not\equiv 0$.\\
{\bf Subcase 1.1.1:} Suppose $\Phi_{3}\not\equiv 0$.\\
First suppose $0$ is not an e.v.P. of $f^{(k)}$ and $g^{(k)}$. Then by {\it Lemma \ref{l2.5}} we get $\Phi_{1}\not\equiv 0$.
Since $f^{(k)}$ and $g^{(k)}$ share $(0,k_{3})$ it follows that $\ol N_{*}(r,0;f^{(k)},g^{(k)})\leq \ol N(r,0;f^{(k)})$. Now successively using {\it Lemmas \it  \ref{l2.19}}, {\it \ref{l2.9}} for $k_{3}=0$, {\it \ref{l2.11}} for $k_{2}=0$ and {\it \ref{l2.14}} we obtain
\bea \label{e3.1} & & \;\;nT(r,f^{(k)})\\&\leq& \ol N(r,\infty;f)+\ol N(r,-a\frac{n-1}{n};f^{(k)})+\ol N(r,\infty;g)+\ol N(r,-a\frac{n-1}{n};g^{(k)})\nonumber\\& & +2\;\ol N(r,0;f^{(k)})+\ol N_{*}(r,0;f^{(k)},g^{(k)})+\ol N_{*}(r,\infty;f,g)-(k_{1}-1)\ol N_{*}(r,1;F,G)\nonumber\\& &+\ol N_{L}(r,1;F)+S(r,f^{(k)})+S(r,g^{(k)})\nonumber\\&\leq& \ol N(r,-a\frac{n-1}{n};f^{(k)})+\ol N(r,-a\frac{n-1}{n};g^{(k)})+3\;\;\ol N(r,0;f^{(k)})+2\ol N(r,\infty;f)\nonumber\\& &+\ol N_{*}(r,\infty;f,g)+S(r,f^{(k)})+S(r,g^{(k)})\nonumber\\&\leq& \ol N(r,-a\frac{n-1}{n};f^{(k)})+\ol N(r,-a\frac{n-1}{n};g^{(k)})+\frac{3k_{1}+3}{(n-2)k_{1}-1}\;\;\ol N(r,\infty;f|\geq k_{2}+1)\nonumber\\ & &+\ol N(r,\infty;f|\geq k_{2}+1) +\frac{2k_{1}+2}{k_{1}k-1}\;\;\;\ol N(r,0;f^{(k)}|\geq k_{3}+1)+S(r,f^{(k)})+S(r,g^{(k)})\nonumber\\&\leq & T(r,f^{(k)})+T(r,g^{(k)})+S(r,f^{(k)})+S(r,g^{(k)})\nonumber\\&\leq & 2\;\;T(r)+S(r).\nonumber\eea
Next suppose $0$ is an e.v.P. of $f^{(k)}$ and $g^{(k)}$. Then $\ol N(r,0;f^{(k)})=S(r,f^{(k)})$.\par Suppose that $\Phi_{1}\not\equiv 0$. Then by {\it Lemma \ref{l2.11}} for $k_{2}=0$ we get $\ol N(r,\infty;f)=S(r)$. So $\ol N_{*}(r,\infty;f,g)=S(r)$.  Consequently (\ref{e3.1}) holds.\par
Next assume $\Phi_{1}\equiv 0$. Then $(F-1)\equiv d(G-1)$, where $d\not=0,1$. Since $f$ and $g$ share $(\infty,k_{2})$, it follows that $f$, $g$ share $(\infty,\infty)$ which implies $\ol N_{*}(r,\infty;f,g)=S(r)$. Also by {\it Lemma \ref{l2.11}} for $k_{2}=0$ we get $\ol N(r,\infty;f)=S(r)$. Clearly in this case also (\ref{e3.1}) holds.\par
In a similar manner as above we can obtain \bea\label{e3.2} nT(r,g^{(k)}) &\leq& 2\;\;T(r)+S(r).\eea
Combining (\ref{e3.1}) and (\ref{e3.2}) we get
\bea\label{e3.3}& &\left(n-2\right)T(r)\leq S(r),\eea
 which leads to a contradiction for $n\geq 3$.\\
{\bf Subcase 1.1.2:} Suppose $\Phi_{3}\equiv 0$. Then by integration we obtain \beas 1-\frac{\omega_{i}}{f^{(k)}}\equiv A(1-\frac{\omega_{j}}{g^{(k)}}),\eeas where $A\not=0$. If $A=1$ then $f^{(k)}=\frac{\omega _{i}}{\omega_{j}} g^{(k)}$, which contradicts $\Phi_{2}\not\equiv 0$. So $A\not=0,1$. Since $f$ and $g$ share $(\infty,k_{3})$, it follows that $N(r,\infty;f)=S(r,f^{(k)})$ and $N(r,\infty;g)=S(r,g^{(k)})$. Now proceeding in the same way as done in the {\it Subcase 1.1.1} we can arrive at a contradiction.\\
{\bf Subcase 1.2:}  Let $\Phi_{2}\equiv 0$.\\
On integration we have $f^{(k)}\equiv cg^{(k)}$, where $c\not=0,1$. Since $f^{(k)}$ and $g^{(k)}$ share $(0,k_{3})$ and $f$, $g$ share  $(\infty,k_{2})$, it follows that $\ol N_{*}(r,0;f^{(k)},g^{(k)})=0$ and $\ol N_{*}(r,\infty;f,g)=0$.\\
{\bf Subcase 1.2.1} Suppose $\Phi_{3}\not\equiv 0$.\\
If $0$ is not an e.v.P. of $f^{(k)}$ and $g^{(k)}$ then by {\it Lemma \ref{l2.5}} we get $\Phi_{1}\not\equiv 0$. Now consecutively using {\it Lemmas {\it \ref{l2.19}}, \ref{l2.16}, \ref{l2.10}} for $k_{2}=0$, and {\it \ref{l2.17}} we obtain
\bea \label{e3.4}& &\;\; nT(r,f^{(k)}) \\&\leq& \ol N(r,\infty;f)+\ol N(r,-a\frac{n-1}{n};f^{(k)})+\ol N(r,\infty;g)+\ol N(r,-a\frac{n-1}{n};g^{(k)})\nonumber\\& & +2\ol N(r,0;f^{(k)})+\ol N_{*}(r,0;f^{(k)},g^{(k)})+\ol N_{*}(r,\infty;f,g)-(k_{1}-1)\;\;\ol N_{*}(r,1;F,G)\nonumber\\ & &+\ol N_{L}(r,1;F) +S(r,f^{(k)})+S(r,g^{(k)})\nonumber\\&\leq& \ol N(r,-a\frac{n-1}{n};f^{(k)})+\ol N(r,-a\frac{n-1}{n};g^{(k)})+2\;\ol N(r,\infty;f)+\frac{2}{k_{1}(n-2)-1}\;\;\ol N(r,\infty;f)\nonumber\\ & & -(k_{1}-1)\;\;\ol N_{*}(r,1;F,G)+\ol N_{L}(r,1;F)+S(r,f^{(k)})+S(r,g^{(k)})\nonumber\\
 &\leq& \ol N(r,-a\frac{n-1}{n};f^{(k)})+\ol N(r,-a\frac{n-1}{n};g^{(k)})+3\;\;\ol N(r,\infty;f)-(k_{1}-1)\ol N_{*}(r,1;F,G)\nonumber\\ & &+\ol N_{L}(r,1;F)+S(r,f^{(k)})+S(r,g^{(k)})\nonumber\\&\leq & 2\;T(r)+\frac{3}{k}\;\;\ol N_{*}(r,1;F,G)-(k_{1}-1)\ol N_{*}(r,1;F,G)+\ol N_{L}(r,1;F)+S(r,f^{(k)})+S(r,g^{(k)})\nonumber\\&\leq & 2\;T(r) +\frac{k_{1}(n-2)}{(k_{1}+1)[k_{1}(n-2)-1]}\;\;\ol N(r,\infty;g)+S(r,f^{(k)})+S(r,g^{(k)})\nonumber\\&\leq & \left(2+\frac{3k_{1}(n-2)}{(k_{1}+1)(k+1)[k_{1}(n-2)-1]}\right)T(r)+S(r).\nonumber\eea
That is
\bea\label{e3.5}& &\left(n-2-\frac{3k_{1}(n-2)}{(k_{1}+1)(k+1)[k_{1}(n-2)-1]}\right)T(r)\leq S(r).\eea
Since $n\geq 3$, (\ref{e3.5}) leads to a contradiction.\par
Suppose $0$ is an e.v.P. of $f^{(k)}$ and $g^{(k)}$. Then {\it Lemma \ref{l2.10}} for $k_{2}=0$ we get $\ol N(r,\infty;f)=\frac{1}{k} \ol N_{*}(r,1;F,G)$. Proceeding as above in this case also we arrive at a contradiction.\\
{\bf Subcase 1.2.2:}  Suppose $\Phi_{3}\equiv 0$.\\   Suppose $\infty$ is not an e.v.P. of $f$ and $g$. Since $f^{(k)}$ and $g^{(k)}$ share $(0,k_{3})$ and $f$,  $g$ share $(\infty,k_{2})$, from {\it Lemma \ref{l2.9a}} it follows that $\ol N_{*}(r,0;f^{(k)},g^{(k)})=0$ and $\ol N_{*}(r,\infty;f,g)=0$.\par
 Suppose $0$ is not an e.v.P of $f^{(k)}$ and $g^{(k)}$ then by {\it Lemma \ref{l2.5}} we get $\Phi_{1}\not\equiv 0$.  Now consecutively using {\it Lemmas \it \ref{l2.19}, \ref{l2.6}} for $k_{3}=0$, {\it \ref{l2.12aa}} for $k_{2}=0$ we obtain
\bea \label{e3.6}& &\;\; nT(r,f^{(k)}) \\&\leq& \ol N(r,\infty;f)+\ol N(r,-a\frac{n-1}{n};f^{(k)})+\ol N(r,\infty;g)+\ol N(r,-a\frac{n-1}{n};g^{(k)})\nonumber\\& & +2\ol N(r,0;f^{(k)})+\ol N_{*}(r,0;f^{(k)},g^{(k)})+\ol N_{*}(r,\infty;f,g)-(k_{1}-1)\;\;\ol N_{*}(r,1;F,G)\nonumber\\ & &+\ol N_{L}(r,1;F) +S(r,f^{(k)})+S(r,g^{(k)})\nonumber\\&\leq& \ol N(r,-a\frac{n-1}{n};f^{(k)})+\ol N(r,-a\frac{n-1}{n};g^{(k)})+2\;\;\ol N(r,\infty;f)+2\;\ol N_{*}(r,1;F,G)\nonumber\\ & & -(k_{1}-1)\;\;\ol N_{*}(r,1;F,G)+\ol N_{L}(r,1;F)+S(r,f^{(k)})+S(r,g^{(k)})\nonumber\\&\leq& \ol N(r,-a\frac{n-1}{n};f^{(k)})+\ol N(r,-a\frac{n-1}{n};g^{(k)})-(k_{1}-3)\ol N_{*}(r,1;F,G)+\ol N_{L}(r,1;F)\nonumber\\ & &+\frac{2}{nk+n-1}\;\{\ol N(r,0;f^{(k)}+a)+\ol N(r,0;g^{(k)}+a)+\ol N_{*}(r,1;F,G)\}+S(r,f^{(k)})+S(r,g^{(k)})\nonumber\\&\leq & 2\;T(r)+\frac{4}{nk+n-1}\;T(r)+ \frac{2}{5}\;\ol N_{L}(r,1;F)+S(r,f^{(k)})+S(r,g^{(k)})\nonumber\\&\leq & \left(2+\frac{4}{nk+n-1}+\frac{2k_{1}(n-2)}{5(k_{1}+1)[k_{1}(n-2)-1]}\right)T(r)+S(r).\nonumber\eea
That is
\bea\label{e3.7}& &\left(n-2-\frac{4}{nk+n-1}-\frac{2k_{1}(n-2)}{5(k_{1}+1)[k_{1}(n-2)-1]}\right)T(r)\leq S(r).\eea
Since $n\geq 3$, (\ref{e3.7}) leads to a contradiction.\par
If $0$ is an e.v.P. of $f^{(k)}$ and $g^{(k)}$ then with the help of {\it Lemmas \ref{l2.19}} and \ref{l2.12aa} for $k_{2}=0$ and proceeding as above we arrive at a contradiction.\\
If $\infty$ is an e.v.P. of $f$ and $g$ then proceeding as in th {\it Subcase 1.2.1} we can arrive at a contradiction.
{\bf Case 2.} Let $H\equiv 0$. Then the theorem follows from {\it Lemma  \ref{l2.24}}.\end{proof}
{

\end{document}